\documentclass[final,leqno]{siamltex}

\usepackage{graphicx}
\usepackage{amsmath,amsfonts}
\usepackage[framemethod=tikz]{mdframed}
\usepackage{showkeys}

\def\beq{\begin{equation}}
\def\eeq{\end{equation}}

\newtheorem{assumption}{Assumption}
\newtheorem{example}{Example}
\def\ba{\begin{array}}
\def\ea{\end{array}}
\def\beann{\begin{eqnarray*}}
\def\eeann{\end{eqnarray*}}
\def\bea{\begin{eqnarray}}
\def\eea{\end{eqnarray}}

\def\BT{\begin{theorem}}
\def\ET{\end{theorem}}
\def\BL{\begin{lemma}}
\def\EL{\end{lemma}}
\def\BC{\begin{corollary}}
\def\EC{\end{corollary}}
\def\BE{\begin{example}}
\def\EE{\end{example}}
\def\BD{\begin{definition}}
\def\ED{\end{definition}}
\def\BR{\begin{remark}}
\def\ER{\end{remark}}
\def\BAS{\begin{assumption}}
\def\EAS{\end{assumption}}
\def\BI{\begin{itemize}}
\def\EI{\end{itemize}}

\def\BMP{\begin{minipage}{9.5cm}}
\def\EMP{\end{minipage}}
\def\MPT{\begin{minipage}{11.5cm}}
\def\EPT{\end{minipage}}

\def\theequation{\thesection.\arabic{equation}}

\newtheorem{remark}[theorem]{Remark}

\title{On the Complexity of an Augmented Lagrangian Method for Nonconvex Optimization}

\author{G.N. Grapiglia\thanks{Departamento de Matem\'atica, Universidade Federal do Paran\'a, Centro Polit\'ecnico, Cx. postal 19.081, 81531-980, Curitiba, Paran\'a, Brazil (grapiglia@ufpr.br). This author was partially supported by the National Council for Scientific and Technological Development - Brazil (grant 406269/2016-5).}
        \and Y. Yuan\thanks{State Key Laboratory of Scientific/Engineering Computing, Institute of Computational
Mathematics and Scientific/Engineering Computing, Academy of Mathematics
and Systems Science, Chinese Academy of Sciences, Zhong Guan Cun Donglu 55, Beijing 100190, People's Republic of China (yyx@lsec.cc.ac.cn). This author was partially supported by NSFC (grant 11688101)}}

\date{July 24, 2019 (version 5)}

\begin{document}

\maketitle

\begin{abstract}
In this paper we study the worst-case complexity of an inexact Augmented Lagrangian method for nonconvex constrained problems. Assuming that the penalty parameters are bounded, we prove a complexity bound of $\mathcal{O}(|\log(\epsilon)|)$ outer iterations for the referred algorithm to generate an $\epsilon$-approximate KKT point, for $\epsilon\in (0,1)$. When the penalty parameters are unbounded, we prove an outer iteration complexity bound of $\mathcal{O}\left(\epsilon^{-2/(\alpha-1)}\right)$, where $\alpha>1$ controls the rate of increase of the penalty parameters. For linearly constrained problems, these bounds yield to evaluation complexity bounds of $\mathcal{O}(|\log(\epsilon)|^{2}\epsilon^{-2})$ and $\mathcal{O}\left(\epsilon^{-\left(\frac{2(2+\alpha)}{\alpha-1}+2\right)}\right)$, respectively, when appropriate first-order methods ($p=1$) are used to approximately solve the unconstrained subproblems at each iteration. In the case of problems having only linear equality constraints, the latter bounds are improved to
$\mathcal{O}(|\log(\epsilon)|^{2}\epsilon^{-(p+1)/p})$ and $\mathcal{O}\left(\epsilon^{-\left(\frac{4}{\alpha-1}+\frac{p+1}{p}\right)}\right)$, respectively, when appropriate $p$-order methods ($p\geq 2$) are used as inner solvers.
\end{abstract}



\pagestyle{myheadings}
\thispagestyle{plain}
\markboth{Complexity of an Augmented Lagrangian Method}{G.N. Grapiglia and Y. Yuan}

\section{Introduction}
\setcounter{equation}{0}

In this paper we consider the constrained optimization problem
\begin{align}
\min_{x\in\mathbb{R}^{n}}&\quad f(x),\label{eq:1.1}\\
\text{s.t.}              &\quad c_{i}(x) =0,\,\,i=1,\ldots,m_{e}, \label{eq:1.2}\\
                         &\quad c_{i}(x)\geq 0,\,\,i=m_{e}+1,\ldots,m, \label{eq:1.3}
\end{align}
where $f,c_{i}:\mathbb{R}^{n}\to\mathbb{R}$ ($i=1,\ldots,m$) are continuously differentiable functions, possibly nonconvex. Augmented Lagrangian Methods are among the most efficient schemes for nonconvex constrained optimization problems (see \cite{HES,POW,BER,BM}). The theoretical analysis of these methods usually focus on global convergence properties. More specifically, in the case of (\ref{eq:1.1})-(\ref{eq:1.3}), for any starting pair $(x_{0},\lambda^{(0)})\in\mathbb{R}^{n}\times\left(\mathbb{R}^{m_{e}}\times\mathbb{R}^{m-m_{e}}_{+}\right)$, one tries to show that the corresponding sequence $\left\{(x_{k},\lambda^{(k)})\right\}_{k\geq 0}$ generated by the Augmented Lagrangian method possess the following asymptotic property (see, e.g., \cite{AND1} and \cite{AND2}):
\\[0.15cm]
\begin{itemize}
\item[\textbf{GC.}] Given $\epsilon>0$, there exists $k_{0}\in\mathbb{N}$ such that
\begin{equation}
\|\nabla f(x_{k_{0}})-\sum_{i=1}^{m}\lambda_{i}^{(k_{0})}\nabla c_{i}(x_{k_{0}})\|\leq\epsilon,
\label{eq:1.4}
\end{equation}
\begin{equation}
\|c_{E}(x_{k_{0}})\|\leq\epsilon,\quad \|c_{I}^{(-)}(x_{k_{0}})\|\leq\epsilon,
\label{eq:1.5}
\end{equation}
\begin{equation}
\lambda_{i}^{(k_{0})}\geq 0,\quad\forall i\in\left\{m_{e}+1,\ldots,m\right\},
\label{eq:1.6}
\end{equation}
and
\begin{equation}
\lambda_{i}^{(k_{0})}=0\quad\text{whenever}\quad c_{i}(x_{k_{0}})>\epsilon,\quad\forall i\in\left\{m_{e}+1,\ldots,m\right\},
\label{eq:1.7}
\end{equation}
where $c_{E}(x)=(c_{1}(x),\ldots,c_{m_{e}}(x))$, $c_{I}(x)=(c_{m_{e}+1}(x),\ldots,c_{m}(x))$, and for any vector $v$, the corresponding vector $v^{(-)}$ is defined by $v_{i}^{(-)}=\min\left\{v_{i},0\right\}$.
\end{itemize}

In this paper we obtain worst-case complexity bounds for the first $k_{0}$ such that (\ref{eq:1.4})-(\ref{eq:1.7}) hold, when $\left\{(x_{k},\lambda^{(k)})\right\}_{k\geq 0}$ is generated by a certain Augmented Lagrangian method that allows inexact solutions of its subproblems. Assuming that the penalty parameters are bounded, we prove a complexity bound of $\mathcal{O}(|\log(\epsilon)|)$ outer iterations for the referred algorithm to generate an $\epsilon$-approximate KKT point, for $\epsilon\in (0,1)$. When the penalty parameters are unbounded, we prove an outer iteration complexity bound of $\mathcal{O}\left(\epsilon^{-2/(\alpha-1)}\right)$, where $\alpha>1$ controls the rate of increase of the penalty parameters. In light of these results, for linearly constrained problems we are able to obtain evaluation complexity bounds of $\mathcal{O}(|\log(\epsilon)|^{2}\epsilon^{-2})$ and $\mathcal{O}\left(\epsilon^{-\left(\frac{2(2+\alpha)}{\alpha-1}+2\right)}\right)$, respectively, when appropriate first-order methods ($p=1$) are used to approximately solve the unconstrained subproblems at each iteration of our Augmented Lagrangian scheme. For problems having only linear equality constraints, these bounds are improved to $\mathcal{O}(|\log(\epsilon)|^{2}\epsilon^{-(p+1)/p})$ and $\mathcal{O}\left(\epsilon^{-\left(\frac{4}{\alpha-1}+\frac{p+1}{p}\right)}\right)$, respectively, when suitable $p$-order methods ($p\geq 2$) are used as inner solvers.

\subsection{Related Literature}
\cite{HONG} proposed a Primal-Dual Algorithm (Prox-PDA) for the linearly constrained problem
\begin{align}
\min_{x\in\mathbb{R}^{n}}&\quad f(x), \label{eq:stela1.1}\\
\text{s.t.}              &\quad a_{i}^{T}x-b_{i}=0,\,\,i=1,\ldots,m.\label{eq:stela1.2}
\end{align}
It was shown that Prox-PDA enjoys an outer iteration complexity of $\mathcal{O}(\epsilon^{-1})$. Therefore, even when the penalty parameters are unbounded, our outer iteration complexity bound of $\mathcal{O}(\epsilon^{-\frac{2}{\alpha-1}})$ (given in Theorem \ref{thm:3.5}) is bether than the one proved in \cite{HONG}, if we take $\alpha>3$. Regarding the evaluation complexity, when the penalty parameters are bounded, our bound of $\mathcal{O}\left(|\log(\epsilon)|^{2}\epsilon^{-2}\right)$ for $p=1$ (given in Theorem \ref{thm:extra3.1}) is slightly worse (in terms of $\epsilon$) than the bounds of $\mathcal{O}(\epsilon^{-2})$ proved for the schemes proposed by \cite{CGT01,CGT0,CGTmais1} and by \cite{BUE}. 

In the context of second-order methods, \cite{CGTmais2}  proposed a Two-Phase algorithm for the general constrained problem (\ref{eq:1.1})-(\ref{eq:1.3}). When the subproblems in their method are solved by a suitable second-order scheme, they proved a complexity bound of $\mathcal{O}\left(\epsilon^{-\frac{3}{2}}\right)$ problem evaluations. Bounds of the same order for different second-order methods were also proved by \cite{Curtis} for equality constrained problems and by \cite{Birgin3} and \cite{HAESER} for linearly constrained problems. Improve complexity bounds of $\mathcal{O}\left(\epsilon^{-\frac{p+1}{p}}\right)$ for $p$-order methods in the constrained case have been obtained by \cite{Birgin2}, \cite{Martinez} and \cite{CGT2,CGT3}. Thus, assuming that the penalty parameters are bounded, our evaluation complexity bound of $\mathcal{O}\left(|\log(\epsilon)|^{2}\epsilon^{-\frac{p+1}{p}}\right)$ (given in Theorem \ref{thm:4.3}) for problem (\ref{eq:stela1.1})-(\ref{eq:stela1.2}) is slightly worse than the bounds mentioned above\footnote{In fact, we can obtain a bound of $\mathcal{O}\left(|\log(\epsilon)|^{2}\epsilon^{-\frac{p+1}{p}}\right)$ even when the penalty parameters are unbounded. This is discussed on Remark 3.2.}.

Finally, very recently, the worst-case complexity of Augmented Lagrangian methods for nonconvex problems have also been investigated by \cite{Birgin5} and \cite{XIE}. Specifically, \cite{Birgin5} prove a bound of $\mathcal{O}(|\log(\epsilon)|)$ outer iterations for the Augmented Lagrangian method corresponding to Algencan (see \cite{AND1}) to find an $\epsilon$-approximate KKT point (or an approximate stationary point of a infeasibility measure, when the penalty parameter goes to infinity). On the other hand, \cite{XIE} analyze the worst-case complexity of a Proximal Augmented Lagrangian method for problems having only equality constraints. Using a constant penalty parameter $\sigma=\mathcal{O}(\epsilon^{-\eta})$, they show that the referred algorithm takes at most $\mathcal{O}(\epsilon^{-(2-\eta)})$ outer iterations to find an $\epsilon$-approximate first-order (or second-order) KKT point, where $\eta$ is an algorithmic parameter chosen in the interval $[0,2]$. Complexity bounds for the total number of inner iterations are also obtained when the Newton-CG method ($p=2$) proposed by \cite{ROY} is used to solve the subproblems. For problems with linear equality constraints and $\eta=2$, they prove evaluations complexity bounds of $\mathcal{O}(\epsilon^{-3/2})$ and $\mathcal{O}(\epsilon^{-3})$ for approximate first-order and second-order optimality, respectively. 


\subsection{Contents} 
The paper is organized as follows. In section 2, we describe the algorithm and obtain outer iteration complexity bounds. In section 3, we obtain worst-case evaluation complexity bounds for linearly constrained problems. Finally, in section 4, we summarize our results and mention some directions for future research. 

\section{Algorithm and Outer Iteration Complexity Analysis}

In what follows we will consider the following assumptions:
\begin{itemize}
\item[\textbf{A1}] Function $f(\,.\,)$ is continuously differentiable, bounded below by $f_{low}$ and, for every $u\in\mathbb{R}$, the set
\begin{equation}
\mathcal{L}_{f}(u)=\left\{x\in\mathbb{R}^{n}\,|\,f(x)\leq u\right\}
\label{eq:stela}
\end{equation}
is bounded.
\item[\textbf{A2}] There exists $\bar{x}\in\mathbb{R}^{n}$ satisfying (\ref{eq:1.2})-(\ref{eq:1.3}) and, for each $i\in\left\{1,\ldots,m\right\}$, function $c_{i}(\,.\,)$ is continuously differentiable.
\end{itemize}

\noindent The Lagrangian function $L(x,\lambda)$ and the Augmented Lagrangian function $P(x,\lambda,\sigma)$ associated to (\ref{eq:1.1})-(\ref{eq:1.3}) are given respectively by:
\begin{equation}
L(x,\lambda)=f(x)-\sum_{i=1}^{m}\lambda_{i}c_{i}(x),
\label{eq:2.3}
\end{equation}
and
\begin{eqnarray}
P(x,\lambda,\sigma)&=&f(x)+\sum_{i=1}^{m_{e}}\left[-\lambda_{i}c_{i}(x)+\frac{1}{2}\sigma c_{i}(x)^{2}\right]\nonumber\\
                   & &+\sum_{i=m_{e}+1}^{m}\left\{\begin{array}{ll}\left[-\lambda_{i}c_{i}(x)+\frac{1}{2}\sigma c_{i}(x)^{2}\right]&\text{if}\,\,c_{i}(x)<\dfrac{\lambda_{i}}{\sigma},\\ -\frac{1}{2}\lambda_{i}^{2}/\sigma, &\text{otherwise}, \end{array}\right.\label{eq:2.4}\\
                   &=& f(x)+\dfrac{\sigma}{2}\sum_{i=1}^{m_{e}}\left[\left(c_{i}(x)-\dfrac{\lambda_{i}}{\sigma}\right)^{2}-\left(\dfrac{\lambda_{i}}{\sigma}\right)^{2}\right]\nonumber\\
                   & &+\dfrac{\sigma}{2}\sum_{i=m_{e}+1}^{m}\left[\left(c_{i}(x)-\dfrac{\lambda_{i}}{\sigma}\right)_{-}^{2}-\left(\dfrac{\lambda_{i}}{\sigma}\right)^{2}\right],\label{eq:2.5} 
\end{eqnarray}
where $\lambda_{i}$ ($i=1,\ldots,m$) are Lagrange multipliers, $\sigma$ is the penalty parameter, and $(\tau)_{-}\equiv \min\left\{0,\tau\right\}$.

Let us consider the following Augmented Lagrangian Method. It is a slight variation of Algorithm 10.4.1 in \cite{SUN}, with a different rule for updating the penalty parameter. Moreover, we allow inexact solutions of the subproblems.
\\[0.2cm]
\begin{mdframed}
\noindent\textbf{Algorithm 1. Augmented Lagrangian Method}
\\[0.2cm]
\noindent\textbf{Step 0.} Given a feasible point $x_{0}\in\mathbb{R}^{n}$ of (\ref{eq:1.2})-(\ref{eq:1.3}), $\lambda^{(0)}\in\mathbb{R}^{m}$ with $\lambda_{i}^{(0)}\geq 0$ for $i=m_{e}+1,\ldots,m$, $\alpha>1$, $\sigma^{(0)}>0$ and $\gamma,\epsilon\in (0,1)$, set $k:=0$.\\
\noindent\textbf{Step 1.} Find an approximate solution $x_{k+1}$ to 
\begin{equation}
\min_{x\in\mathbb{R}^{n}}\,P(x,\lambda^{(k)},\sigma^{(k)}),
\label{eq:2.6}
\end{equation}
such that
\begin{equation}
P(x_{k+1},\lambda^{(k)},\sigma^{(k)})\leq\min\left\{P(x_{k},\lambda^{(k)},\sigma^{(k)}),P(x_{0},\lambda^{(k)},\sigma^{(k)})\right\},
\label{eq:2.7}
\end{equation}
and
\begin{equation}
\|\nabla_{x}P(x_{k+1},\lambda^{(k)},\sigma^{(k)})\|_{\infty}\leq\epsilon.
\label{eq:2.8}
\end{equation}
\noindent\textbf{Step 2.} Compute 
\begin{equation}
\theta^{(k+1)}=\max\left\{\left\|\frac{\lambda^{(k)}}{\sigma^{(k)}}\right\|_{\infty},\left\|c_{E}(x_{k+1})-\frac{\lambda_{E}^{(k)}}{\sigma^{(k)}}\right\|_{\infty},\|c_{I}^{(-)}(x_{k+1})\|_{\infty}\right\}, 
\label{eq:mais2.8}
\end{equation}
where $\lambda_{E}^{(k)}=(\lambda_{1}^{(k)},\ldots,\lambda_{m_{e}}^{(k)})$. If $k=0$, set $\sigma^{(k+1)}=\sigma^{(k)}$. Otherwise, set
\begin{equation}
\sigma^{(k+1)}=\left\{\begin{array}{ll} \sigma^{(k),}&\text{if}\quad \theta^{(k+1)}\leq\gamma\theta^{(k)},\\
\max\left\{(k+1)^{\alpha},\sigma^{(k)}\right\},&\text{otherwise.}
\end{array}
\right.
\label{eq:2.9}
\end{equation}
\noindent\textbf{Step 3.} Set 
\begin{equation}
\lambda_{i}^{(k+1)}=\lambda_{i}^{(k)}-\sigma^{(k)}c_{i}(x_{k+1}),\,\,\text{for}\,\,i=1,\ldots,m_{e},
\label{eq:extra2.10}
\end{equation}
and
\begin{equation}
\lambda_{i}^{(k+1)}=\max\left\{\lambda_{i}^{(k)}-\sigma^{(k)}c_{i}(x_{k+1}),0\right\}\,\,\text{for}\,\,i=m_{e}+1,\ldots,m.
\label{eq:2.10}
\end{equation}
\noindent\textbf{Step 4.} Set $k:=k+1$ and go back to Step 1.
\end{mdframed}

In this section, our goal is to establish iteration complexity bounds for Algorithm 1, that is, upper bounds on the number of outer iterations necessary to generate a pair $(x_{k},\lambda^{(k)})$ such that
\begin{equation}
\|\nabla_{x}L(x_{k},\lambda^{(k)})\|_{\infty}\leq\epsilon,
\label{eq:3.1}
\end{equation}
\begin{equation}
\|c_{E}(x_{k})\|_{\infty}\leq\epsilon,\quad\|c_{I}^{(-)}(x_{k})\|_{\infty}\leq\epsilon,
\label{eq:mais3.2}
\end{equation}
\begin{equation}
\lambda_{i}^{(k)}\geq 0,\quad\forall i\in\left\{m_{e}+1,\ldots,m\right\},
\label{eq:stela3.3}
\end{equation}
and
\begin{equation}
\lambda_{i}^{(k)}=0\quad\text{whenever}\quad c_{i}(x_{k})>\epsilon,\quad\forall i\in\left\{m_{e}+1,\ldots,m\right\},
\label{eq:mais3.3}
\end{equation}
\noindent for a given $\epsilon\in (0,1)$. For that, we need some auxiliary results. The first one states that the Augmented Lagrangian function is bounded from above by the objective function on points of the feasible set.

\begin{lemma}
\label{lem:3.1}
Suppose that A2 holds. Let $\sigma>0$, $\lambda\in\mathbb{R}^{m}$ with $\lambda_{i}\geq 0$ for $i=m_{e}+1,\ldots,m$, and $\bar{x}\in\mathbb{R}^{n}$ be a feasible point of (\ref{eq:1.2})-(\ref{eq:1.3}). Then,
\begin{equation}
P(\bar{x},\lambda,\sigma)\leq f(\bar{x}).
\label{eq:3.2}
\end{equation}
\end{lemma}
\begin{proof}
By assumptions, for $i=m_{e}+1,\ldots,m$, we have $c_{i}(\bar{x})\geq 0$ and $\dfrac{\lambda_{i}}{\sigma}\geq 0$. Thus, 
\begin{equation*}
c_{i}(\bar{x})-\dfrac{\lambda_{i}}{\sigma}\geq -\dfrac{\lambda_{i}}{\sigma}.
\end{equation*}
Consequently,
\begin{equation*}
0\geq \left(c_{i}(\bar{x})-\dfrac{\lambda_{i}}{\sigma}\right)_{-}\geq -\dfrac{\lambda_{i}}{\sigma},
\end{equation*}
which shows that
\begin{equation*}
\left(c_{i}(\bar{x})-\dfrac{\lambda_{i}}{\sigma}\right)_{-}^{2}\leq\left(\dfrac{\lambda_{i}}{\sigma}\right)^{2},\,\,i=m_{e}+1,\ldots,m.
\end{equation*}
Then, (\ref{eq:3.2}) follows from the above inequality, the equality $c_{i}(\bar{x})=0$ for $i=1,\ldots,m_{e}$ and (\ref{eq:2.5}). 
\end{proof}

The lemma below gives a lower bound for the Augmented Lagrangian function.
\begin{lemma}
Let $\sigma>0$ and $\lambda\in\mathbb{R}^{m}$ with $\lambda_{i}\geq 0$ for $i=m_{e}+1,\ldots,m$. Then,
\begin{equation*}
P(x,\lambda,\sigma)\geq f(x)-\dfrac{1}{2}\sum_{i=1}^{m}\dfrac{\left[\lambda_{i}\right]^{2}}{\sigma},\quad\forall x\in\mathbb{R}^{n}.
\end{equation*}
\label{lem:plus2.2}
\end{lemma}
\begin{proof}
By (\ref{eq:2.5}), we have
\begin{eqnarray*}
P(x,\lambda,\sigma)&\geq & f(x)-\dfrac{\sigma}{2}\sum_{i=1}^{m_{e}}\left(\dfrac{\lambda_{i}}{\sigma}\right)^{2}-\dfrac{\sigma}{2}\sum_{i=m_{e}+1}^{m}\left(\dfrac{\lambda_{i}}{\sigma}\right)^{2}\\
& = & f(x)-\dfrac{1}{2}\sum_{i=1}^{m}\dfrac{\left[\lambda_{i}\right]^{2}}{\sigma},
\end{eqnarray*}
for all $x\in\mathbb{R}^{n}$.
\end{proof}

In view of Lemmas \ref{lem:3.1} and \ref{lem:plus2.2}, we can show that the iterates of Algorithm 1 are well defined.
\begin{lemma}
Suppose that A1 and A2 hold. Given $\sigma>0$ and $\lambda\in\mathbb{R}^{m}$, with $\lambda_{i}\geq 0$ for $i=m_{e}+1,\ldots,m$, let $P(\,.\,,\lambda,\sigma)$ be defined by (\ref{eq:2.5}). Then, there exists at least one point $x^{+}$ such that 
\begin{equation*}
\nabla_{x} P(x^{+},\lambda,\sigma)=0.
\end{equation*}
\label{lem:plus2.3}
\end{lemma}
\begin{proof}
By A1 and A2, it follows that $P(\,.\,,\lambda,\sigma)$ is differentiable. Thus, it is enough to show that $P(\,.\,,\lambda,\sigma)$ has a global minimizer. For that, given $\bar{x}\in\mathbb{R}^{n}$ feasible, we will show that the set 
\begin{equation*}
\mathcal{L}_{P}(\bar{x})=\left\{z\in\mathbb{R}^{n}\,|\,P(z,\lambda,\sigma)\leq P(\bar{x},\lambda,\sigma)\right\}
\end{equation*}
is compact (then the existence of a global minimizer will follows from Weierstrass Theorem). Indeed, consider $x\in\mathcal{L}_{p}(\bar{x})$. Then, by Lemmas \ref{lem:3.1} and \ref{lem:plus2.2} we have
\begin{equation*}
f(x)-\dfrac{1}{2}\sum_{i=1}^{m}\dfrac{\left[\lambda_{i}\right]^{2}}{\sigma}\leq P(x,\lambda,\sigma)\leq P(\bar{x},\lambda,\sigma)\leq f(\bar{x})
\end{equation*}
\begin{equation*}
\Longrightarrow f(x)\leq f(\bar{x})+\dfrac{1}{2}\sum_{i=1}^{m}\dfrac{\left[\lambda_{i}\right]^{2}}{\sigma},
\end{equation*}
and so, $x\in\mathcal{L}_{f}(u)$ for $u=f(\bar{x})+\frac{1}{2}\sum_{i=1}^{m}\left[\lambda_{i}\right]^{2}/\sigma$, where $\mathcal{L}_{f}(u)$ is defined in (\ref{eq:stela}). This means that $\mathcal{L}_{P}(\bar{x})\subset\mathcal{L}_{f}(u)$. Since $\mathcal{L}_{f}(u)$ is bounded (due to A1), it follows that $\mathcal{L}_{P}(\bar{x})$ is also bounded. Then, by the continuity of $P(\,.\,,\lambda,\sigma)$ we conclude that $\mathcal{L}_{P}(\bar{x})$ is compact.
\end{proof}

The next lemma provides an upper bound for the constraint violation on infeasible points.
\begin{lemma}
\label{lem:3.2}
Let $\sigma>0$, $\lambda\in\mathbb{R}^{m}$ with $\lambda_{i}\geq 0$ for $i=m_{e}+1,\ldots,m$, and $\bar{x}\in\mathbb{R}^{n}$ be a feasible point of (\ref{eq:1.2})-(\ref{eq:1.3}). If $x^{+}$ is an infeasible point of (\ref{eq:1.2})-(\ref{eq:1.3}) such that  
\begin{equation}
P(x^{+},\lambda,\sigma)\leq P(\bar{x},\lambda,\sigma),
\label{eq:3.4}
\end{equation}
then
\begin{equation}
\dfrac{1}{2}\left(\sum_{i=1}^{m}\dfrac{\lambda_{i}^{2}}{\sigma}\right)+f(\bar{x})-f(x^{+})\geq\dfrac{\sigma}{2} \max\left\{\left\|c_{E}(x^{+})-\frac{\lambda_{E}}{\sigma}\right\|_{\infty},\|c_{I}^{(-)}(x^{+})\|_{\infty}\right\}^{2},
\label{eq:3.5}
\end{equation}
where $\lambda_{E}=(\lambda_{1},\ldots,\lambda_{m_{e}})$.
\end{lemma}
\begin{proof}
From (\ref{eq:3.4}), Lemma \ref{lem:3.1} and (\ref{eq:2.5}), it follows that
\begin{eqnarray*}
0&\leq & P(\bar{x},\lambda,\sigma)-P(x^{+},\lambda,\sigma)\\
 &\leq & f(\bar{x})-f(x^{+})-\dfrac{\sigma}{2}\sum_{i=1}^{m_{e}}\left[\left(c_{i}(x^{+})-\dfrac{\lambda_{i}}{\sigma}\right)^{2}-\left(\dfrac{\lambda_{i}}{\sigma}\right)^{2}\right]\\
 &     &-\dfrac{\sigma}{2}\sum_{i=m_{e}+1}^{m}\left[\left(c_{i}(x^{+})-\dfrac{\lambda_{i}}{\sigma}\right)^{2}_{-}-\left(\dfrac{\lambda_{i}}{\sigma}\right)^{2}\right]\\
 & = & f(\bar{x})-f(x^{+})+\dfrac{\sigma}{2}\sum_{i=1}^{m}\left(\dfrac{\lambda_{i}}{\sigma}\right)^{2}-\dfrac{\sigma}{2}\sum_{i=1}^{m_{e}}\left(c_{i}(x^{+})-\dfrac{\lambda_{i}}{\sigma}\right)^{2}\\
 &   &-\dfrac{\sigma}{2}\sum_{i=m_{e}+1}^{m}\left(c_{i}(x^{+})-\dfrac{\lambda_{i}}{\sigma}\right)^{2}_{-}.
\end{eqnarray*}
Therefore,
\begin{equation}
\dfrac{1}{2}\left(\sum_{i=1}^{m}\dfrac{\lambda_{i}^{2}}{\sigma}\right)+f(\bar{x})-f(x^{+})\geq\dfrac{\sigma}{2}\max\left\{\left\|c_{E}(x^{+})-\frac{\lambda_{E}}{\sigma}\right\|_{2}^{2},\sum_{i=m_{e}+1}^{m}\left(c_{i}(x^{+})-\dfrac{\lambda_{i}}{\sigma}\right)^{2}_{-}\right\}.
\label{eq:geo1}
\end{equation}
Let $J=\left\{j\in\left\{m_{e}+1,\ldots,m\right\}\,|\,c_{j}(x^{+})<0\right\}$. Since $x^{+}$ is an infeasible point to (\ref{eq:1.2})-(\ref{eq:1.3}), we may have $J\neq\emptyset$. In this case,
\begin{equation}
\sum_{i=m_{e}+1}^{m}\left(c_{i}(x^{+})-\dfrac{\lambda_{i}}{\sigma}\right)^{2}_{-}\geq \sum_{i\in J}\left(c_{i}(x^{+})-\dfrac{\lambda_{i}}{\sigma}\right)^{2}_{-}= \sum_{i\in J}\left(c_{i}(x^{+})-\dfrac{\lambda_{i}}{\sigma}\right)^{2}.
\label{eq:stela1}
\end{equation}
For $i\in J$, we have
\begin{equation*}
c_{i}(x^{+})-\dfrac{\lambda_{i}}{\sigma}\leq c_{i}(x^{+})<0,
\end{equation*}
which gives
\begin{eqnarray}
\sum_{i\in J}\left(c_{i}(x^{+})-\dfrac{\lambda_{i}}{\sigma}\right)^{2}&\geq&\sum_{i\in J}c_{i}(x^{+})^{2}=\sum_{i=m_{e}+1}^{m}\min\left\{c_{i}(x^{+}),0\right\}^{2}\nonumber\\
& = &\sum_{i=m_{e}+1}^{m}c_{i}^{(-)}(x^{+})^{2}\nonumber\\
& = &\|c^{(-)}_{I}(x^{+})\|_{2}^{2}.
\label{eq:stela2}
\end{eqnarray}
Thus, combining (\ref{eq:geo1}), (\ref{eq:stela1}), (\ref{eq:stela2}) and the norm inequality $\|\,.\,\|_{2}\geq\|\,.\,\|_{\infty}$, we get (\ref{eq:3.5}).

\end{proof}

The following lemma specializes the upper bound from Lemma \ref{lem:3.2} to points generated by Algorithm 1.
\begin{lemma}
\label{lem:3.3}
Suppose that A1 and A2 hold, and let the sequence $\left\{(x_{k},\lambda^{(k)},\sigma^{(k)})\right\}_{k\geq 0}$ be generated by Algorithm 1. Then, for all $k\geq 1$, we have
\begin{equation}
k\left[\|\mu^{(0)}\|_{2}^{2}+4(f(x_{0})-f_{low})\right]\geq\sigma^{(k)}\max\left\{\left\|c_{E}(x_{k+1})-\frac{\lambda_{E}^{(k)}}{\sigma^{(k)}}\right\|_{\infty},\|c_{I}^{(-)}(x_{k+1})\|_{\infty}\right\}^{2},
\label{eq:3.6}
\end{equation}
where
\begin{equation}
\mu_{i}^{(k)}=\dfrac{\lambda_{i}^{(k)}}{\sqrt{\sigma^{(k)}}},\,\,i=1,\ldots,m.
\label{eq:3.7}
\end{equation}
\end{lemma}
\begin{proof}
Since $x_{0}$ is a feasible point of (\ref{eq:1.2})-(\ref{eq:1.3}), it follows from (\ref{eq:2.7}), Lemma \ref{lem:3.2}, (\ref{eq:3.7}) and the bound $f(x_{k+1})\geq f_{low}$ that
\begin{equation}
\dfrac{1}{2}\|\mu^{(k)}\|_{2}^{2}+f(x_{0})-f_{low}\geq\dfrac{\sigma^{(k)}}{2}\max\left\{\left\|c_{E}(x_{k+1})-\frac{\lambda_{E}^{(k)}}{\sigma^{(k)}}\right\|_{\infty},\|c_{I}^{(-)}(x_{k+1})\|_{\infty}\right\}^{2}.
\label{eq:3.8}
\end{equation}
Now, let us compute an upper bound for $\|\mu^{(k)}\|_{2}^{2}$. From (\ref{eq:3.7}), (\ref{eq:2.9}), (\ref{eq:extra2.10}), (\ref{eq:2.10}), (\ref{eq:2.4}), Lemma \ref{lem:3.1} and (\ref{eq:2.7}), we have
\begin{eqnarray*}
\|\mu^{(k+1)}\|_{2}^{2}& = &\sum_{i=1}^{m}\dfrac{[\lambda_{i}^{(k+1)}]^{2}}{\sigma^{(k+1)}}\leq\sum_{i=1}^{m}\dfrac{[\lambda_{i}^{(k+1)}]^{2}}{\sigma^{(k)}}\\
& = & \sum_{i=1}^{m}\dfrac{[\lambda_{i}^{(k)}]^{2}}{\sigma^{(k)}}+\sum_{i=1}^{m_{e}}\dfrac{\left(\lambda_{i}^{(k)}-\sigma^{(k)}c_{i}(x_{k+1})\right)^{2}-[\lambda_{i}^{(k)}]^{2}}{\sigma^{(k)}}\\
&   &+\sum_{i=m_{e}+1}^{m}\dfrac{\left(\max\left\{\lambda_{i}^{(k)}-\sigma^{(k)}c_{i}(x_{k+1}),0\right\}\right)^{2}-[\lambda_{i}^{(k)}]^{2}}{\sigma^{(k)}}\\
& = & \|\mu^{(k)}\|_{2}^{2}+2\sum_{i=1}^{m_{e}}\left[-\lambda_{i}^{(k)}c_{i}(x_{k+1})+\frac{1}{2}\sigma^{(k)}c_{i}(x_{k+1})^{2}\right]\\
&   &+2\sum_{i=m_{e}+1}^{m}\left\{\begin{array}{ll} \left[-\lambda_{i}^{(k)}c_{i}(x_{k+1})+\frac{1}{2}\sigma^{(k)}c_{i}(x_{k+1})^{2}\right],&\text{if}\,\,c_{i}(x_{k+1})<\frac{\lambda_{i}^{(k)}}{\sigma^{(k)}}\\
-\dfrac{[\lambda_{i}^{(k)}]^{2}}{2\sigma^{(k)}},&\text{otherwise}
                                                  \end{array}
                                           \right.\\
& = & \|\mu^{(k)}\|_{2}^{2}+2\left[P(x_{k+1},\lambda^{(k)},\sigma^{(k)})-f(x_{k+1})\right]\\
&\leq & \|\mu^{(k)}\|_{2}^{2}+2\left[P(x_{k+1},\lambda^{(k)},\sigma^{(k)})-f(x_{k+1})\right]-2\left[P(x_{0},\lambda^{(k)},\sigma^{(k)})-f(x_{0})\right]\\
& = & \|\mu^{(k)}\|_{2}^{2}+2\left[P(x_{k+1},\lambda^{(k)},\sigma^{(k)})-P(x_{0},\lambda^{(k)},\sigma^{(k)})\right]+2\left[f(x_{0})-f(x_{k+1})\right]\\
&\leq & \|\mu^{(k)}\|_{2}^{2}+2\left[f(x_{0})-f_{low}\right].
\end{eqnarray*}
Therefore, by induction, 
\begin{equation}
\|\mu^{(k)}\|_{2}^{2}\leq \|\mu^{(0)}\|_{2}^{2}+2\left(f(x_{0})-f_{low}\right)k
\label{eq:3.9}
\end{equation}
Finally, combining (\ref{eq:3.8}) and (\ref{eq:3.9}) we get (\ref{eq:3.6}).
\end{proof}

Our last auxiliary result gives an upper bound for $\theta^{(k)}$.
\begin{lemma}
\label{lem:mais3.3}
Suppose that A1 and A2 hold, and let the sequence $\left\{(x_{k},\lambda^{(k)},\sigma^{(k)})\right\}_{k\geq 0}$ be generated by Algorithm 1. Then 
\begin{equation}
k\left[\|\mu^{(0)}\|_{2}^{2}+4(f(x_{0})-f_{low})\right]\geq\sigma^{(k)}\left[\theta^{(k+1)}\right]^{2},\quad\forall k\geq 1,
\label{eq:mais3.6}
\end{equation}
where $\theta^{(k)}$ is defined in (\ref{eq:mais2.8}).
\end{lemma}

\begin{proof}
By (\ref{eq:3.7}), we have
\begin{equation}
\|\mu^{(k)}\|_{2}\geq\|\mu^{(k)}\|_{\infty}=\left\|\frac{\lambda^{(k)}}{\sqrt{\sigma^{(k)}}}\right\|_{\infty}=\left\|\frac{\lambda^{(k)}}{\sigma^{(k)}}\right\|_{\infty}\sqrt{\sigma^{(k)}}.
\label{eq:mais3.7}
\end{equation}
Thus, combining (\ref{eq:mais3.7}) and (\ref{eq:3.9}), it follows that
\begin{eqnarray}
\sigma^{(k)}\left\|\frac{\lambda^{(k)}}{\sigma^{(k)}}\right\|_{\infty}^{2}&\leq &\|\mu^{(k)}\|_{2}^{2}\nonumber\\
                                         &\leq & \|\mu^{(0)}\|_{2}^{2}+2(f(x_{0})-f_{low})k\nonumber\\
                                         &\leq &\left[\|\mu^{(0)}\|_{2}^{2}+4\left(f(x_{0})-f_{low}\right)\right]k.
\label{eq:geo2}
\end{eqnarray}
Finally, combining (\ref{eq:geo2}), (\ref{eq:3.6}) and (\ref{eq:mais2.8}), we obtain (\ref{eq:mais3.6}).
\end{proof}

Now, we will analyse the outer iteration complexity of Algorithm 1 considering separately the following cases:
\begin{itemize}
\item[(i)] $\left\{\sigma^{(k)}\right\}_{k\geq 0}$ is bounded from above by $\sigma_{max}$.
\item[(ii)] $\lim_{k\to +\infty}\sigma^{(k)}=+\infty$.
\end{itemize}

By (\ref{eq:2.3}), (\ref{eq:extra2.10}), (\ref{eq:2.10}), (\ref{eq:2.4}) and (\ref{eq:2.8}), we have
\begin{eqnarray*}
\|\nabla_{x}L(x_{k},\lambda^{(k)})\|_{\infty}&=&\left\|\nabla f(x_{k})-\sum_{i=1}^{m}\lambda_{i}^{(k)}\nabla c_{i}(x_{k})\right\|_{\infty}\\
& = & \left\|\nabla f(x_{k})-\sum_{i=1}^{m_{e}}\left(\lambda_{i}^{(k-1)}-\sigma^{(k-1)}c_{i}(x_{k})\right)\nabla c_{i}(x_{k})\right.\\
&   &\left.-\sum_{i=m_{e}+1}^{m}\max\left\{\lambda_{i}^{(k-1)}-\sigma^{(k-1)}c_{i}(x_{k}),0\right\}\nabla c_{i}(x_{k})\right\|_{\infty}\\
& = & \|\nabla_{x}P(x_{k},\lambda^{(k-1)},\sigma^{(k-1)})\|_{\infty}\\
& \leq & \epsilon,
\end{eqnarray*}
for all $k\geq 1$. Thus, to bound the number of outer iterations necessary to ensure (\ref{eq:3.1})-(\ref{eq:mais3.3}), we only need to bound the number of iterations in which (\ref{eq:mais3.2}) or (\ref{eq:mais3.3}) does not hold. Note that, if 
\begin{equation}
\theta^{(k)}\leq\epsilon
\label{eq:mais3.8}
\end{equation}
and $c_{i}(x_{k})>\epsilon$, then
\begin{equation*}
\dfrac{\lambda_{i}^{(k-1)}}{\sigma^{(k-1)}}-c_{i}(x_{k})<\epsilon-\epsilon=0,
\end{equation*}
and so
\begin{equation*}
\lambda_{i}^{(k-1)}-\sigma^{(k-1)}c_{i}(x_{k})<0.
\end{equation*}
Therefore, if (\ref{eq:mais3.8}) holds, by (\ref{eq:2.10}) we have
\begin{equation*}
\lambda_{i}^{(k)}=0\quad\text{whenever}\quad c_{i}(x_{k})>\epsilon,\quad\forall i\in\left\{m_{e}+1,\ldots,m\right\}.
\end{equation*}
On the other hand, it is easy to see that
\begin{equation*}
\theta^{(k)}\leq\dfrac{\epsilon}{2}\Longrightarrow \max\left\{\|c_{E}(x_{k})\|_{\infty},\|c_{I}^{(-)}(x_{k})\|_{\infty}\right\}\leq\epsilon.
\end{equation*}
In view of these remarks, all we have to do is to bound the number of iterations in which
\begin{equation*}
\theta^{(k)}>\dfrac{\epsilon}{2}.
\end{equation*}

The theorem below gives an upper bound of $\mathcal{O}\left(|\log(\epsilon)|\right)$ iterations in case (i).
\begin{theorem}
\label{thm:3.4}
Suppose that A1 and A2 hold, and let the sequence $\left\{(x_{k},\lambda^{(k)},\sigma^{(k)})\right\}_{k\geq 0}$ be generated by Algorithm 1. Moreover, assume that $\left\{\sigma^{(k)}\right\}_{k\geq 0}$ is bounded from above by $\sigma_{max}$. If 
\begin{equation}
T>\sigma_{max}^{\frac{1}{\alpha}}+2+\dfrac{\frac{1}{2}\log\left(\|\mu^{(0)}\|_{2}^{2}+4\left[f(x_{0})-f_{low}\right]\right)+\log(2)+|\log(\epsilon)|}{\log(\gamma^{-1})}
\label{eq:3.11}
\end{equation}
then
\begin{equation}
\theta^{(T)}\leq \dfrac{\epsilon}{2},
\label{eq:3.10}
\end{equation}
where $\mu^{(0)}$ is defined by (\ref{eq:3.7}) and $\alpha$ and $\gamma$ are the parameters in (\ref{eq:2.9}).
\end{theorem}
\begin{proof}
Consider the set 
\begin{equation}
\mathcal{U}=\left\{j\in\mathbb{N}\,|\,\theta^{(j+1)}>\gamma\theta^{(j)}\right\}.
\label{eq:3.12}
\end{equation}
Given $k\in \mathcal{U}$, we must have $k\leq\sigma_{max}^{\frac{1}{\alpha}}$, since otherwise, by (\ref{eq:2.9}) we would have $\sigma^{(k+1)}\geq (k+1)^{\alpha}>\sigma_{max}$, contradicting our assumption on the penalty parameters. 

Let us denote $\bar{k}=\max\left\{j\,|\,j\in\mathcal{U}\right\}$. By (\ref{eq:3.11}), we have $T=\bar{k}+2+s$ with
\begin{equation}
s\geq\dfrac{\frac{1}{2}\log\left(\|\mu^{(0)}\|_{2}^{2}+4\left[f(x_{0})-f_{low}\right]\right)+\log(2)+|\log(\epsilon)|}{\log(\gamma^{-1})}.
\label{eq:eug1}
\end{equation}
Note that
\begin{equation*}
\left\{\bar{k}+2,\ldots,\bar{k}+2+s\right\}=\left\{\bar{k}+2,\ldots,T\right\}\subset\mathbb{N}-\mathcal{U}.
\end{equation*}
Consequently,
\begin{equation}
\theta^{(T)}=\theta^{(\bar{k}+2+s)}\leq\gamma^{s}\theta^{(\bar{k}+2)}.
\label{eq:3.13}
\end{equation}
Since $\bar{k}\in\mathcal{U}$, it follows that $\sigma^{(\bar{k}+1)}\geq (\bar{k}+1)^{\alpha}$. Thus, by Lemma 2.4 we have
\begin{equation*}
(\bar{k}+1)^{\alpha}\left[\theta^{(\bar{k}+2)}\right]^{2}\leq\left[\|\mu^{(0)}\|_{2}^{2}+4(f(x_{0})-f_{low})\right](\bar{k}+1)
\end{equation*}
\begin{eqnarray*}
\Longrightarrow \,\, \left[\theta^{(\bar{k}+2)}\right]^{2}&\leq & \left[\|\mu^{(0)}\|_{2}^{2}+4\left(f(x_{0})-f_{low}\right)\right] \dfrac{(\bar{k}+1)}{(\bar{k}+1)^{\alpha}}\\
&\leq & \left[\|\mu^{(0)}\|_{2}^{2}+4\left(f(x_{0})-f_{low}\right)\right] 
\end{eqnarray*}
\begin{equation}
\Longrightarrow \,\, \theta^{(\bar{k}+2)}\leq \left(\|\mu^{(0)}\|_{2}^{2}+4\left[f(x_{0})-f_{low}\right]\right)^{\frac{1}{2}}.
\label{eq:3.14}
\end{equation}
Combining (\ref{eq:3.13}) and (\ref{eq:3.14}), we obtain
\begin{equation*}
\theta^{(T)}\leq\gamma^{s}\left(\|\mu^{(0)}\|_{2}^{2}+4\left[f(x_{0})-f_{low}\right]\right)^{\frac{1}{2}}
\end{equation*}
By (\ref{eq:eug1}),
\begin{equation*}
\gamma^{s}\left(\|\mu^{(0)}\|_{2}^{2}+4\left[f(x_{0})-f_{low}\right]\right)^{\frac{1}{2}}\leq\dfrac{\epsilon}{2}.
\end{equation*}
Therefore, 
\begin{equation*}
\theta^{(T)}\leq\dfrac{\epsilon}{2},
\end{equation*}
and the proof is complete.
\end{proof}

The next theorem establishes an upper bound of $\mathcal{O}\left(\epsilon^{-\frac{2}{\alpha-1}}\right)$ iterations in case (ii).

\begin{theorem}
\label{thm:3.5}
Suppose that A1 and A2 hold, and let the sequence $\left\{(x_{k},\lambda^{(k)},\sigma^{(k)})\right\}_{k\geq 0}$ be generated by Algorithm 1 such that
\begin{equation}
\theta^{(k)}>\frac{\epsilon}{2},\quad\text{for}\,\,k=1,\ldots,T.
\label{eq:3.15}
\end{equation}
If $\lim_{k\to +\infty}\sigma^{(k)}=+\infty$, then
\begin{equation}
T< 4+\left(4\|\mu^{(0)}\|_{2}^{2}+16[f(x_{0})-f_{low}]\right)^{\frac{1}{\alpha-1}}\epsilon^{-\frac{2}{\alpha-1}}+\dfrac{\frac{1}{2}\log\left(\|\mu^{(0)}\|_{2}^{2}+4\left[f(x_{0})-f_{low}\right]\right)+\log(2)+|\log(\epsilon)|}{\log(\gamma^{-1})}
\label{eq:3.16}
\end{equation}
where $\mu^{(0)}$ is defined by (\ref{eq:3.7}) and $\alpha$ and $\gamma$ are the parameters in (\ref{eq:2.9}).
\end{theorem}
\begin{proof}
Again, consider the set $\mathcal{U}$ defined in (\ref{eq:3.12}). Given $k\in\mathcal{U}$, it follows from (\ref{eq:2.9}) and Lemma \ref{lem:mais3.3} that
\begin{equation*}
(k+1)\left(\|\mu^{(0)}\|_{2}^{2}+4[f(x_{0})-f_{low}]\right)\geq (k+1)^{\alpha}\left[\theta^{(k+2)}\right]^{2}
\end{equation*}
\begin{equation*}
\Longrightarrow \left[\theta^{(k+2)}\right]^{2}\leq\left(\|\mu^{(0)}\|_{2}^{2}+4[f(x_{0})-f_{low}]\right)\dfrac{1}{(k+1)^{\alpha-1}}.
\end{equation*}
Thus, 
\begin{equation*}
k\in\mathcal{U}\,\,\text{and}\,\,k+1\geq\left(4\|\mu^{(0)}\|_{2}^{2}+16[f(x_{0})-f_{low}]\right)^{\frac{1}{\alpha-1}}\epsilon^{-\frac{2}{\alpha-1}}\Longrightarrow \theta^{(k+2)}\leq\dfrac{\epsilon}{2}.
\end{equation*}
By (\ref{eq:3.15}), this means that
\begin{equation*}
\tilde{k}=\max\left\{k\in\left\{1,\ldots,T-2\right\}\,|\,k\in\mathcal{U}\right\}\leq \left(4\|\mu^{(0)}\|_{2}^{2}+16[f(x_{0})-f_{low}]\right)^{\frac{1}{\alpha-1}}\epsilon^{-\frac{2}{\alpha-1}}.
\end{equation*}
Suppose that $T-2>\tilde{k}+2$ and define $s=(T-2)-(\tilde{k}+2)$. Then,
\begin{equation*}
\left\{\tilde{k}+2,\ldots,\tilde{k}+2+s\right\}=\left\{\tilde{k}+2,\ldots,T-2\right\}\subset\left\{1,\ldots,T-2\right\}-\mathcal{U},
\end{equation*}
and so
\begin{equation}
\theta^{(T-2)}=\theta^{(\tilde{k}+2+s)}\leq\gamma^{s}\theta^{(\tilde{k}+2)}.
\label{eq:3.17}
\end{equation}
Since $\tilde{k}\in\mathcal{U}$, it follows that $\sigma^{(\tilde{k}+1)}\geq (\tilde{k}+1)^{\alpha}$. Thus, using Lemma \ref{lem:mais3.3} as in the proof of Theorem \ref{thm:3.4}, we obtain 
\begin{equation}
\theta^{(\tilde{k}+2)}\leq \left(\|\mu^{(0)}\|_{2}^{2}+4\left[f(x_{0})-f_{low}\right]\right)^{\frac{1}{2}}.
\label{eq:3.18}
\end{equation}
Combining (\ref{eq:3.17}) and (\ref{eq:3.18}), we obtain
\begin{equation*}
\theta^{(T-2)}\leq \gamma^{s}\left(\|\mu^{(0)}\|_{2}^{2}+4\left[f(x_{0})-f_{low}\right]\right)^{\frac{1}{2}}.
\end{equation*}
Note that, if 
\begin{equation*}
s\geq\dfrac{\frac{1}{2}\log\left(\|\mu^{(0)}\|_{2}^{2}+4\left[f(x_{0})-f_{low}\right]\right)+\log(2)+|\log(\epsilon)|}{\log(\gamma^{-1})},
\end{equation*}
we would have
\begin{equation*}
\theta^{(T-2)}\leq \dfrac{\epsilon}{2},
\end{equation*}
contradicting (\ref{eq:3.15}). 
Hence, we must have
\begin{equation*}
s<\dfrac{\frac{1}{2}\log\left(\|\mu^{(0)}\|_{2}^{2}+4\left[f(x_{0})-f_{low}\right]\right)+|\log(\epsilon)|}{\log(\gamma^{-1})}.
\end{equation*}
Therefore,
\begin{eqnarray*}
T-2=\tilde{k}+2+s&<&\left(4\|\mu^{(0)}\|_{2}^{2}+16[f(x_{0})-f_{low}]\right)^{\frac{1}{\alpha-1}}\epsilon^{-\frac{2}{\alpha-1}}+2\\
& &+\dfrac{\frac{1}{2}\log\left(\|\mu^{(0)}\|_{2}^{2}+4\left[f(x_{0})-f_{low}\right]\right)+\log(2)+|\log(\epsilon)|}{\log(\gamma^{-1})},
\end{eqnarray*}
which gives (\ref{eq:3.16}).
\end{proof}

In summary, Theorem \ref{thm:3.4} means that if the sequence of penalty parameters is bounded from above, then Algorithm 1 takes at most $\mathcal{O}(|\log(\epsilon)|)$ outer iterations to generate $(x_{k},\lambda^{(k)})$ satisfying (\ref{eq:3.1})-(\ref{eq:mais3.3}), that is, an $\epsilon$-approximate KKT point of (\ref{eq:1.1})-(\ref{eq:1.3}). On the other hand, if the sequence of penalty parameters is unbounded, Theorem \ref{thm:3.5} gives an upper bound of $\mathcal{O}(\epsilon^{-\frac{2}{\alpha-1}})$ outer iterations, where $\alpha>1$ defines the increase of the penalty parameter. As expected (see \cite{BER}, p. 104), the bigger $\alpha$ is, the more aggressive the update of the penalty parameter is, and consequently, the maximum number of outer iterations needed for Algorithm 1 to find and $\epsilon$-approximate KKT point is smaller.

\section{Evaluation Complexity for Linearly Constrained Problems}

At each outer iteration of Algorithm 1, subproblem (\ref{eq:2.6}) must be approximately solved. In general, this is done by applying an iterative optimization method for unconstrained problems. When its invoked, this unconstrained method requires a certain number of calls of the oracle\footnote{By \textit{calls of the oracle} we mean the joint computation of $f$, $c_{i}$ ($i=1,\ldots,m$) and their derivatives.}, which is proportional to the number of inner iterations. Therefore, a full estimation of the worst-case complexity of Algorithm 1 must also take into account the iteration complexiy of the auxiliary unconstrained method used to solve the subproblems. With this goal in mind, now we shall consider the following special case of (\ref{eq:1.1})-(\ref{eq:1.3}):
\begin{align}
\min_{x\in\mathbb{R}^{n}}&\quad f(x),\label{eq:extra4.1}\\
\text{s.t.}              &\quad a_{i}^{T}x-b_{i}=0,\,\,i=1,\ldots,m_{e} \label{eq:extra4.2}\\
                         &\quad a_{i}^{T}x-b_{i}\geq 0,\,\,i=m_{e}+1,\ldots,m,\label{eq:extra4.3},
\end{align}
where $f:\mathbb{R}^{n}\to\mathbb{R}$ is continuously differentiable, possibly nonconvex, $a_{i}\in\mathbb{R}^{n}$ and $b_{i}\in\mathbb{R}$ for $i=1,\ldots,m$. Given $\lambda\in\mathbb{R}^{m}$, with $\lambda_{i}\geq 0$ for $i=m_{e}+1,\ldots,m$, and $\sigma>0$, the corresponding augmented Lagrangian function to (\ref{eq:extra4.1})-(\ref{eq:extra4.3}) is 
\begin{eqnarray}
P(x,\lambda,\sigma)&=&f(x)+\sum_{i=1}^{m_{e}}\left[-\lambda_{i}(a_{i}^{T}x-b_{i})+\frac{1}{2}\sigma (a_{i}^{T}x-b_{i})^{2}\right]\nonumber\\
& &+\sum_{i=m_{e}+1}^{m}\left\{\begin{array}{ll} \left[-\lambda_{i}(a_{i}^{T}x-b_{i})+\frac{1}{2}\sigma (a_{i}^{T}x-b_{i})^{2}\right],&\text{if}\,\,a_{i}^{T}x-b_{i}<\frac{\lambda_{i}}{\sigma},\\
-\frac{1}{2}\lambda_{i}^{2}/\sigma,&\text{otherwise}.
                        \end{array}
                        \right.
\label{eq:extra4.4}
\end{eqnarray}
For each $i\in\left\{1,\ldots,n\right\}$, let us denote by $a_{ij}$ the $j$-th component of vector $a_{i}$. Specifically, let us consider the following additional assumption on $f$:
\begin{itemize}
\item[\textbf{A3}] The gradient of $f(\,.\,)$, $\nabla f:\mathbb{R}^{n}\to\mathbb{R}^{n}$, is $L_{1}$-Lipschitz continuous, i.e., 
\begin{equation*}
\|\nabla f(x)-\nabla f(y)\|\leq L_{1}\|x-y\|,\quad\forall x,y\in\mathbb{R}^{n}.
\end{equation*}
\end{itemize}

\begin{lemma}
Given $\lambda\in\mathbb{R}^{n}$, with $\lambda_{i}\geq 0$ for $i=m_{e}+1,\ldots,m$, and $\sigma>0$, let function $P(\,.\,,\lambda,\sigma)$ be defined by (\ref{eq:extra4.4}). Suppose that A3 holds and that
\begin{equation}
a_{ij}\geq 0,\quad i=m_{e}+1,\ldots,m,\,\,j=1,\ldots,n.
\label{eq:extra4.5}
\end{equation}
Then,
\begin{equation}
\|\nabla P(x,\lambda,\sigma)-\nabla P(y,\lambda,\sigma)\|_{2}\leq\sqrt{n}\left(L_{1}+\sigma\|A\|_{F}^{2}\right)\|x-y\|_{2},\quad\forall x,y\in\mathbb{R}^{n},
\label{eq:extra4.6}
\end{equation}
where the $i$-th row of $A\in\mathbb{R}^{m\times n}$ is vector $a_{i}^{T}$ and $\|\,.\,\|_{F}$ denotes the Frobenius norm.
\label{lem:extra3.1}
\end{lemma}

\begin{proof}
Indeed, from (\ref{eq:extra4.4}), a direct calculation shows that
\begin{eqnarray*}
\dfrac{\partial P}{\partial x_{j}}(x,\lambda,\sigma)&=&\dfrac{\partial f}{\partial x_{j}}(x)+\sum_{i=1}^{m_{e}}\left[-\lambda_{i}a_{ij}+\sigma(a_{i}^{T}x-b_{i})a_{ij}\right]\\
& &+\sum_{i=m_{e}+1}^{m}\left\{\begin{array}{ll} \left[-\lambda_{i}a_{ij}+\sigma(a_{i}^{T}x-b_{i})a_{ij}\right],&\text{if}\,\,a_{i}^{T}x-b_{i}<\frac{\lambda_{i}}{\sigma},\\ 0,&\text{otherwise}.
                        \end{array}
                        \right.
\end{eqnarray*}
Thus, given $x,y\in\mathbb{R}^{n}$, we have
\begin{equation}
\dfrac{\partial P}{\partial x_{j}}(x,\lambda,\sigma)-\dfrac{\partial P}{\partial x_{j}}(y,\lambda,\sigma)=\dfrac{\partial f}{\partial x_{j}}(x)-\dfrac{\partial f}{\partial x_{j}}(y)+\sum_{i=1}^{m_{e}}\sigma [a_{i}^{T}(x-y)]a_{ij}+\sum_{i=m_{e}+1}^{m}\beta_{i}(x,y),
\label{eq:extra4.7}
\end{equation}
where
\begin{equation*}
\beta_{i}(x,y)=\left\{\begin{array}{ll} \sigma [a_{i}^{T}(x-y)]a_{ij},&\text{if}\,\,\max\left\{a_{i}^{T}x-b_{i},a_{i}^{T}y-b_{i}\right\}<\frac{\lambda_{i}}{\sigma},\\
                                    &                     \\
-\lambda_{i}a_{ij}+\sigma(a_{i}^{T}x-b_{i})a_{ij},&\text{if}\,\,a_{i}^{T}x-b_{i}<\frac{\lambda_{i}}{\sigma}\leq a_{i}^{T}y-b_{i},\\
&                \\
\lambda_{i}a_{ij}-\sigma(a_{i}^{T}y-b_{i})a_{ij},&\text{if}\,\,a_{i}^{T}y-b_{i}<\frac{\lambda_{i}}{\sigma}\leq a_{i}^{T}x-b_{i},
                      \end{array}
                \right.
\end{equation*}
for $i=m_{e}+1,\ldots,m$. If $\max\left\{a_{i}^{T}x-b_{i},a_{i}^{T}y-b_{i}\right\}<\frac{\lambda_{i}}{\sigma}$, then
\begin{equation*}
\beta_{i}(x,y)=\sigma [a_{i}^{T}(x-y)]a_{ij}\leq\sigma \|a_{i}\|_{2}\|x-y\|_{2}|a_{ij}|.
\end{equation*}
If $a_{i}^{T}x-b_{i}<\frac{\lambda_{i}}{\sigma}\leq a_{i}^{T}y-b_{i}$, then (\ref{eq:extra4.5}) implies that
\begin{equation*}
\sigma(a_{i}^{T}x-b_{i})a_{ij}\leq\lambda_{i}a_{ij},
\end{equation*} 
and so
\begin{equation*}
\beta_{i}(x,y)=-\lambda_{i}a_{ij}+\sigma(a_{i}^{T}x-b_{i})a_{ij}\leq 0\leq \sigma\|a_{i}\|_{2}\|x-y\|_{2}|a_{ij}|.
\end{equation*}
Finally, if $a_{i}^{T}y-b_{i}<\frac{\lambda_{i}}{\sigma}\leq a_{i}^{T}x-b_{i}$, then (\ref{eq:extra4.5}) implies that
\begin{equation*}
\lambda a_{ij}\leq\sigma (a_{i}^{T}x-b_{i})a_{ij},
\end{equation*}
and so
\begin{eqnarray*}
\beta_{i}(x,y)&=&\lambda_{i}a_{ij}-\sigma (a_{i}^{T}y-b_{i})a_{ij}\leq\sigma (a_{i}^{T}x-b_{i})a_{ij}-\sigma (a_{i}^{T}y-b_{i})a_{ij}\\
              &= & \sigma [a_{i}^{T}(x-y)]a_{ij}\\
              &\leq & \sigma\|a_{i}\|_{2}\|x-y\|_{2}|a_{ij}|.
\end{eqnarray*}
In summary,
\begin{equation}
\beta_{i}(x,y)\leq\sigma\|a_{i}\|_{2}\|x-y\|_{2}|a_{ij}|,\quad i=m_{e}+1,\ldots,m.
\label{eq:extra4.8}
\end{equation}
Combining (\ref{eq:extra4.7}) and (\ref{eq:extra4.8}), we obtain
\begin{equation*}
\dfrac{\partial P}{\partial x_{j}}(x,\lambda,\sigma)-\dfrac{\partial P}{\partial x_{j}}(y,\lambda,\sigma)\leq\left|\dfrac{\partial f}{\partial x_{j}}(x)-\dfrac{\partial f}{\partial x_{j}}(y)\right|+\sum_{i=1}^{m}\sigma\|a_{i}\|_{2}\|x-y\|_{2}|a_{ij}|.
\end{equation*}
Similarly, we also have
\begin{equation*}
\dfrac{\partial P}{\partial x_{j}}(y,\lambda,\sigma)-\dfrac{\partial P}{\partial x_{j}}(x,\lambda,\sigma)\leq\left|\dfrac{\partial f}{\partial x_{j}}(y)-\dfrac{\partial f}{\partial x_{j}}(x)\right|+\sum_{i=1}^{m}\sigma\|a_{i}\|_{2}\|y-x\|_{2}|a_{ij}|.
\end{equation*}
Therefore,
\begin{eqnarray}
\left|\dfrac{\partial P}{\partial x_{j}}(x,\lambda,\sigma)-\dfrac{\partial P}{\partial x_{j}}(y,\lambda,\sigma)\right|&\leq&\left|\dfrac{\partial f}{\partial x_{j}}(x)-\dfrac{\partial f}{\partial x_{j}}(y)\right|+\sum_{i=1}^{m}\sigma\|a_{i}\|_{2}\|x-y\|_{2}|a_{ij}|\nonumber\\
&\leq & \left|\dfrac{\partial f}{\partial x_{j}}(x)-\dfrac{\partial f}{\partial x_{j}}(y)\right|+\sigma\|A\|_{F}^{2}\|x-y\|_{2}.
\label{eq:extra4.9}
\end{eqnarray}
From (\ref{eq:extra4.9}) and A3, we obtain
\begin{eqnarray*}
\|\nabla P(x,\lambda,\sigma)-\nabla P(y,\lambda,\sigma)\|_{2}&\leq &\sqrt{n}\|\nabla P(x,\lambda,\sigma)-\nabla P(y,\lambda,\sigma)\|_{\infty}\\
&\leq &\sqrt{n}\left(\|\nabla f(x)-\nabla f(y)\|_{\infty}+\sigma\|A\|_{F}^{2}\|x-y\|_{2}\right)\\
&\leq &\sqrt{n}\left(\|\nabla f(x)-\nabla f(y)\|_{2}+\sigma\|A\|_{F}^{2}\|x-y\|_{2}\right)\\
&\leq &\sqrt{n}\left(L_{1}+\sigma\|A\|_{F}^{2}\right)\|x-y\|_{2},
\end{eqnarray*}
that is, (\ref{eq:extra4.6}) holds.
\end{proof}

\begin{remark}
Recall that any linearly constrained problem can be formulated in the standard form
\begin{align}
\min_{x\in\mathbb{R}^{n}}&\quad f(x),\label{eq:extra4.10}\\
\text{s.t.}              &\quad a_{i}^{T}x-b_{i}=0,\,\,i=1,\ldots,m_{e} \label{eq:extra4.11}\\
                         &\quad x_{j}\geq 0,\,\,j=1,\ldots,n,\label{eq:extra4.12}
\end{align}
for which assumption (\ref{eq:extra4.5}) holds.
\end{remark}

From Lemma \ref{lem:extra3.1} it follows that $\nabla P(\,.\,,\lambda,\sigma)$ is Lipschitz continuous when $\nabla f(\,.\,)$ is Lipschitz continuous. In this case, one can minimize $P(\,.\,,\lambda,\sigma)$ by using a first-order method $\mathcal{M}_{1}$. More specifically, let us consider the following assumption on $\mathcal{M}_{1}$:
\begin{itemize}
\item[\textbf{A4}] Given any differentiable function $g:\mathbb{R}^{n}\to\mathbb{R}$, bounded below by $g_{low}$, with $L$-Lipschitz gradient, method $\mathcal{M}_{1}$ with starting point $\tilde{x}$, can find an $\epsilon$-approximate stationary point of $g(\,.\,)$ in at most
\begin{equation*}
C_{\mathcal{M}_{1}}L(g(\tilde{x})-g_{low})\epsilon^{-2}
\end{equation*}
iterations, where $C_{\mathcal{M}_{1}}$ is a positive constant that depends only on the method $\mathcal{M}_{1}$.
\end{itemize}
Examples of methods that satisfy A4 can be found in \cite{NES}, \cite{GST}, \cite{CGT} and \cite{GYY}. The next lemma establishes an upper bound on the number of iterations for these methods to compute approximate solutions of (\ref{eq:2.6}) satisfying (\ref{eq:2.7})-(\ref{eq:2.8}).

\begin{lemma}
Suppose that A1 and A3 hold, and let $\left\{(x_{j},\lambda^{(j)},\sigma^{(j)})\right\}_{j=0}^{k}$ be generated by Algorithm 1 applied to (\ref{eq:extra4.1})-(\ref{eq:extra4.3}), with $\left\{a_{i}\right\}_{i=m_{e}+1}^{m}$ satisfying (\ref{eq:extra4.5}). Moreover, assume that a monotone method $\mathcal{M}_{1}$ is applied to minimize $P(\,.\,,\lambda^{(k)},\sigma^{(k)})$ with starting point
\begin{equation}
\tilde{x}_{k,0}=\arg\min\left\{P(y,\lambda^{(k)},\sigma^{(k)})\,|\,y\in\left\{x_{0},x_{k}\right\}\right\}.
\label{eq:extra4.13}
\end{equation}
If method $\mathcal{M}_{1}$ satisfies A4, then $\mathcal{M}_{1}$ takes at most $\mathcal{O}\left(\max\left\{1,\sigma^{(k)}\right\}(k+1)\sqrt{n}\epsilon^{-2}\right)$ iterations to generate $\tilde{x}_{k,\ell}$ such that
\begin{equation*}
\|\nabla_{x}P(\tilde{x}_{k,\ell},\lambda^{(k)},\sigma^{(k)})\|_{2}\leq\epsilon.
\end{equation*}
\label{lem:extra3.2}
\end{lemma}

\begin{proof}
By Lemma \ref{lem:plus2.2}, A1 and (\ref{eq:3.9}), for all $x\in\mathbb{R}^{n}$ we have
\begin{eqnarray}
P(x,\lambda^{(k)},\sigma^{(k)})&\geq & f(x)-\dfrac{1}{2}\sum_{i=1}^{m}\dfrac{[\lambda_{i}^{(k)}]^{2}}{\sigma^{(k)}}\nonumber\\
& = & f(x)-\dfrac{1}{2}\|\mu^{(k)}\|_{2}^{2}\nonumber\\
&\geq & f_{low}-\dfrac{1}{2}\|\mu^{(0)}\|_{2}^{2}-(f(x_{0})-f_{low})k\equiv P_{low}.
\label{eq:extra4.14}
\end{eqnarray}
On the other hand, it follows from (\ref{eq:extra4.13}) and Lemma \ref{lem:3.1} that 
\begin{equation}
P(\tilde{x}_{k,0},\lambda^{(k)},\sigma^{(k)})\leq P(x_{0},\lambda^{(k)},\sigma^{(k)})\leq f(x_{0}).
\label{eq:extra4.15}
\end{equation}
Finally, from A3 and Lemma \ref{lem:extra3.1}, it follows that $\nabla P(\,.\,,\lambda^{(k)},\sigma^{(k)})$ is Lipschitz continuous with constant
\begin{equation*}
L=\max\left\{1,\sigma^{(k)}\right\}\sqrt{n}\left(L_{1}+\|A\|_{F}^{2}\right).
\end{equation*}
Thus, by A4, (\ref{eq:extra4.14}) and (\ref{eq:extra4.15}), we conclude that method $\mathcal{M}_{1}$ takes at most
\begin{equation*}
C_{\mathcal{M}_{1}}\max\left\{1,\sigma^{(k)}\right\}\sqrt{n}\left(L_{1}+\|A\|_{F}^{2}\right)\left(f(x_{0})-f_{low}+\frac{1}{2}\|\mu^{(0)}\|_{2}^{2}\right)(k+1)\epsilon^{-2}
\end{equation*}
iterations to generate an $\epsilon$-approximate stationary point of $P(\,.\,,\lambda^{(k)},\sigma^{(k)})$.
\end{proof}

Now, combining Theorems \ref{thm:3.4} and \ref{thm:3.5} with Lemma \ref{lem:extra3.2} we can obtain worst-case complexity bounds for the total number of inner iterations performed in Algorithm 1 to find an $\epsilon$-approximate KKT point of (\ref{eq:extra4.1})-(\ref{eq:extra4.3}).

\begin{theorem}
Suppose that Algorithm 1 is applied to solve (\ref{eq:extra4.1})-(\ref{eq:extra4.3}), with $f$ satisfying A1 and A3, and $\left\{a_{i}\right\}_{m_{e}+1}^{m}$ satisfying (\ref{eq:extra4.5}). Moreover, assume that at each outer iteration of Algorithm 1, a monotone method $\mathcal{M}_{1}$ satisfying A4 is used to approximately solve (\ref{eq:2.6}) with starting point $\tilde{x}_{k,0}$ given in (\ref{eq:extra4.13}). Then, the following statements are true:
\begin{itemize}
\item[(a)] If $\left\{\sigma^{(k)}\right\}_{k\geq 0}$ is bounded from above by $\sigma_{max}$, then Algorithm 1 takes at most $\mathcal{O}\left(\sigma_{max}\sqrt{n}|\log(\epsilon)|^{2}\epsilon^{-2}\right)$ inner iterations of $\mathcal{M}_{1}$ to generate and $\epsilon$-approximate KKT point of (\ref{eq:extra4.1})-(\ref{eq:extra4.3}).
\item[(b)] If $\lim_{k\to +\infty}\sigma^{(k)}=+\infty$ and $\alpha\geq 2$ is integer, then Algorithm 1 takes at most $\mathcal{O}\left(\sqrt{n}\epsilon^{-\left(\frac{2(2+\alpha)}{\alpha-1}+2\right)}\right)$ iterations of $\mathcal{M}_{1}$ to generate an $\epsilon$-approximate KKT point of (\ref{eq:extra4.1})-(\ref{eq:extra4.3}).
\end{itemize}
\label{thm:extra3.1}
\end{theorem}

\begin{proof}
By A1 and Theorems \ref{thm:3.4} and \ref{thm:3.5}, there exists an iteration number $T$ such that (\ref{eq:3.1})-(\ref{eq:mais3.3}) hold for $k=T+1$. Without loss of generality, assume that $\sigma^{(0)}=1$. Then, $\sigma^{(k)}\geq 1$ and, by Lemma \ref{lem:extra3.2}, the total number of inner iterations of $\mathcal{M}_{1}$ performed until iteration $T+1$ of Algorithm 1 is proportional to 
\begin{equation*}
\left(\sum_{k=1}^{T+2}\sigma^{(k)}k\right)\sqrt{n}\epsilon^{-2}.
\end{equation*}
If $\left\{\sigma^{(k)}\right\}_{k\geq 0}$ is bounded from above by $\sigma_{max}$, we have
\begin{equation*}
\left(\sum_{k=1}^{T+2}\sigma^{(k)}k\right)\sqrt{n}\epsilon^{-2}\leq\left(\sum_{k=1}^{T+2}k\right)\sigma_{max}\sqrt{n}\epsilon^{-2}=\mathcal{O}\left((T+2)^{2}\sigma_{max}\sqrt{n}\epsilon^{-2}\right).
\end{equation*}
Thus, by Theorem \ref{thm:3.4}, we get a bound of $\mathcal{O}\left(|\log(\epsilon)|^{2}\sigma_{max}\sqrt{n}\epsilon^{-2}\right)$ for the total number of inner iterations of $\mathcal{M}_{1}$, proving statement (a).
\\
Suppose that $\lim_{k\to +\infty}\sigma^{(k)}=+\infty$. By (\ref{eq:2.9}) and $\sigma^{(0)}=1$ we have $\sigma^{(k)}\leq k^{\alpha}$ for all $k$. Since $\alpha$ is a positive integer, we have
\begin{equation*}
\left(\sum_{k=1}^{T+2}\sigma^{(k)}k\right)\sqrt{n}\epsilon^{-2}\leq\left(\sum_{k=1}^{T+2}k^{1+\alpha}\right)\sqrt{n}\epsilon^{-2}=\mathcal{O}\left((T+2)^{2+\alpha}\sqrt{n}\epsilon^{-2}\right),
\end{equation*}
where the last equality is due to the fact that $\sum_{k=1}^{T+2}k^{1+\alpha}=\mathcal{O}((T+2)^{2+\alpha})$ (see, e.g., \cite{SCHULTZ}). Thus, by Theorem \ref{thm:3.5}, we get a bound of $\mathcal{O}\left(\sqrt{n}\epsilon^{-\left(\frac{2(2+\alpha)}{\alpha-1}+2\right)}\right)$ for the total number of inner iterations of $\mathcal{M}_{1}$, proving statement (b).
\end{proof}

Now, let us consider linearly constrained problems of the form
\begin{align}
\min_{x\in\mathbb{R}^{n}}&\quad f(x),\label{eq:4.1}\\
\text{s.t.}              &\quad a_{i}^{T}x-b_{i}=0,\,\,i=1,\ldots,m, \label{eq:4.2}
\end{align}
where $f:\mathbb{R}^{n}\to\mathbb{R}$ is $p$-times differentiable ($p\geq 2$), possibly nonconvex. This is a particular case of (\ref{eq:1.1})-(\ref{eq:1.3}) with $m=m_{e}$. The corresponding augmented Lagrangian function to (\ref{eq:4.1})-(\ref{eq:4.2}) is

\begin{equation}
P(x,\lambda,\sigma)=f(x)+\sum_{i=1}^{m}\left[-\lambda_{i}(a_{i}^{T}x-b_{i})+\frac{1}{2}\sigma (a_{i}^{T}x-b_{i})^{2}\right].
\label{eq:4.3}
\end{equation}
We denote by $D^{p}f(x)$ the $p$-th order derivative of $f$ at point $x$, which is a tensor defined by
\begin{equation*}
\left[D^{p}f(x)\right]_{i_{1}\ldots i_{p}}=\dfrac{\partial^{p}f}{\partial x_{i_{1}}\ldots\partial x_{i_{p}}}(x),\,\,1\leq i_{1},\ldots,i_{p}\leq n.
\end{equation*}
In particular $D^{2}f(x)=\nabla^{2}f(x)$. The next lemma relates the $p$-th order derivatives of $P(\,.\,,\lambda,\sigma)$ and $f(\,.\,)$.
\begin{lemma}
\label{lem:4.1}
Given $f:\mathbb{R}^{n}\to\mathbb{R}$, $\lambda\in\mathbb{R}^{m}$ and $\sigma>0$, let $P(x,\lambda,\sigma)$ be defined by (\ref{eq:4.3}). If $f$ is $p$-times differentiable with $p\geq 2$, then
\begin{equation}
D^{p}P(x,\lambda,\sigma)=\left\{\begin{array}{ll} \nabla^{2}f(x)+\sigma A^{T}A,&\text{if}\,\,p=2,\\
                                                  D^{p}f(x),&\text{if}\,\,p\geq 3,
                               \end{array}
                         \right.
\label{eq:4.4}
\end{equation}
where the $i$-th row of $A\in\mathbb{R}^{m\times n}$ is vector $a_{i}^{T}$.
\end{lemma}
\begin{proof}
Indeed, from (\ref{eq:4.3}), a direct calculation shows that
\begin{equation*}
\dfrac{\partial P}{\partial x_{j}}(x,\lambda,\sigma)=\dfrac{\partial f}{\partial x_{j}}(x)+\sum_{i=1}^{m}\left[-\lambda_{i}a_{ij}+\sigma (a_{i}^{T}x-b_{i})a_{ij}\right].
\end{equation*}
Then,
\begin{equation*}
\dfrac{\partial^{2} P}{\partial x_{k}\partial x_{j}}(x,\lambda,\sigma)=\dfrac{\partial^{2} f}{\partial x_{k}\partial x_{j}}(x)+\sum_{i=1}^{m}\sigma a_{ik}a_{ij},
\end{equation*}
which gives
\begin{equation*}
\dfrac{\partial^{3} P}{\partial x_{\ell}\partial x_{k}\partial x_{j}}(x,\lambda,\sigma)=\dfrac{\partial^{3} f}{\partial x_{\ell}\partial x_{k}\partial x_{j}}(x).
\end{equation*}
Therefore, (\ref{eq:4.4}) holds.
\end{proof}

For $p\geq 2$, it follows from Lemma \ref{lem:4.1} that $D^{p}P(\,.\,,\lambda,\sigma)$ is $L_{p}$-Lipschitz continuous when $D^{p}f(\,.\,)$ is $L_{p}$-Lipschitz continuous. In this case, one can minimize $P(\,.\,,\lambda,\sigma)$ by using a $p$-order method $\mathcal{M}_{p}$. More specifically, let us consider the following assumptions on $f$ and $\mathcal{M}_{p}$:
\begin{itemize}
\item[\textbf{A5.}] $D^{p}f(\,.\,)$ is $L_{p}$-Lipschitz continuous, i.e.,  
\begin{equation*}
\|D^{p}f(x)-D^{p}f(y)\|\leq L_{p}\|x-y\|,\quad\forall x,y\in\mathbb{R}^{n}.
\end{equation*}
\item[\textbf{A6.}] Given any $p$-times differentiable function $g:\mathbb{R}^{n}\to\mathbb{R}$, bounded below by $g_{low}$, with $L$-Lipschitz continous $p$-th order derivative ($p\geq 2$), method $\mathcal{M}_{p}$ with starting point $\tilde{x}$, can find an $\epsilon$-approximate stationary point of $g(\,.\,)$ in at most
\begin{equation*}
C_{\mathcal{M}_{p}}L^{\frac{1}{p}}\left(g(\tilde{x})-g_{low}\right)\epsilon^{-\frac{p+1}{p}}
\end{equation*}
iterations, where $C_{\mathcal{M}_{p}}$ is a positive constant that depends only on the method $\mathcal{M}_{p}$. 
\end{itemize}

An important class of unconstrained methods that satisfy A6 for $p=2$ is the class of methods based on the cubic regularization of Newton's method (see, e.g., \cite{NP,CGT,Birgin,GN}). Moreover, for $p\geq 3$, several tensor methods satisfying A6 have been proposed recently (see, e.g., \cite{Birgin,CGT2,GN2}).The use of these tensor methods to approximately solve (\ref{eq:2.6}) give us complexity bounds better than the ones obtained in Theorem \ref{thm:extra3.1}, as we can see below.
\begin{theorem}
\label{thm:4.3}
Suppose that Algorithm 1 is applied to solve (\ref{eq:4.1})-(\ref{eq:4.2}) with $f$ satisfying A1 and A5. Moreover, assume that at each outer iteration of Algorithm 1, a monotone method $\mathcal{M}_{p}$ satisfying A6 is used to approximately solve (\ref{eq:2.6}) with starting point $\tilde{x}_{k,0}$ given in (\ref{eq:extra4.13}). Then, the following statements are true:
\begin{itemize}
\item[(a)] If $\left\{\sigma^{(k)}\right\}_{k\geq 0}$ is bounded from above by $\sigma_{max}$, then Algorithm 1 takes at most $\mathcal{O}\left(|\log(\epsilon)|^{2}\epsilon^{-\frac{p+1}{p}}\right)$ inner iterations of $\mathcal{M}_{p}$ to generate an $\epsilon$-approximate KKT point of (\ref{eq:4.1})-(\ref{eq:4.2}).
\item[(b)] If $\lim_{k\to +\infty}\sigma^{(k)}=+\infty$, then Algorithm 1 takes at most $\mathcal{O}\left(\epsilon^{-\left(\frac{4}{\alpha-1}+\frac{p+1}{p}\right)}\right)$ inner iterations of $\mathcal{M}_{p}$ to generate an $\epsilon$-approximate KKT point of (\ref{eq:4.1})-(\ref{eq:4.2}).
\end{itemize}
\end{theorem}

\begin{proof}
As in Lemma \ref{lem:extra3.2}, by A1, A5 and A6 we conclude that, at the $k$-th iteration, method $\mathcal{M}_{p}$ takes at most $\mathcal{O}\left(k\epsilon^{-\frac{p+1}{p}}\right)$ iterations to generate $\tilde{x}_{k,\ell}$ such that
\begin{equation*}
\|\nabla_{x}P(\tilde{x}_{k,\ell},\lambda^{(k)},\sigma^{(k)})\|_{2}\leq\epsilon.
\end{equation*}
Thus, the total number of inner iterations of $\mathcal{M}_{p}$ performed until iteration $T+1$ is proportional to 
\begin{equation*}
\left(\sum_{k=1}^{T+2}k\right)\epsilon^{-\frac{p+1}{p}}=\mathcal{O}\left((T+2)^{2}\epsilon^{-\frac{p+1}{p}}\right).
\end{equation*}
Consequently, (a) and (b) follow directly from the upper bounds on $T$ given in (\ref{eq:3.11}) and (\ref{eq:3.16}).
\end{proof}

\begin{remark}
It is worth mentioning that a complexity bound of $\mathcal{O}\left(|\log(\epsilon)|^{2}\epsilon^{-\frac{p+1}{p}}\right)$ can be obtained even when the penalty parameters are unbounded. For that, it is enough to replace (\ref{eq:2.9}) by the following update rule:
\begin{equation}
\sigma^{(k+1)}=\left\{\begin{array}{ll} \sigma^{(k),}&\text{if}\quad\theta^{(k+1)}\leq\gamma\theta^{(k)},\\
\max\left\{4^{k+1},\sigma^{(k)}\right\},&\text{otherwise.}
\end{array}
\right.
\label{eq:5.4}
\end{equation}
Indeed, using (\ref{eq:5.4}) in the proof of Theorem \ref{thm:3.5}, we can obtain $\tilde{k}\leq\mathcal{O}\left(|\log(\epsilon)|\right)$ which gives an outer ite\-ra\-tion complexity bound of $\mathcal{O}\left(|\log(\epsilon)|\right)$. This corresponds to an evaluation complexity bound of $\mathcal{O}\left(|\log(\epsilon)|^{2}\epsilon^{-\frac{p+1}{p}}\right)$ for problem (\ref{eq:4.1})-(\ref{eq:4.2}) when the penalty parameters are unbounded. However, (\ref{eq:5.4}) increases the penalty parameters much more aggressively than (\ref{eq:2.9}), and can lead to a premature ill-conditioning of (\ref{eq:2.6}). Therefore, despite the improved worst-case complexity, it is unlikely that (\ref{eq:5.4}) will give an efficient algorithm in practice.
\end{remark}

\begin{remark}
Consider now the following problem
\begin{align}
\min_{x\in\mathbb{R}^{n}}&\quad f(x),\label{eq:5.5}\\
\text{s.t.}              &\quad x^{T}R_{i}x+s_{i}^{T}x+t_{i} = 0,\,\,i=1,\ldots,m\label{eq:5.6}
\end{align}
where $f:\mathbb{R}^{n}\to\mathbb{R}$ is four times differentiable, possibly nonconvex, $R_{i}\in\mathbb{R}^{n\times n}$, $s_{i}\in\mathbb{R}^{n}$ and $t_{i}\in\mathbb{R}$, for $i=1,\ldots,m$. Let $P(x,\lambda,\sigma)$ denote the augmented Lagrangian function corresponding to (\ref{eq:5.5})-(\ref{eq:5.6}). Then, as in Lemma \ref{lem:4.1}, one can see that $D^{5}P(x,\lambda,\sigma)=D^{5}f(x)$ for all $x\in\mathbb{R}^{n}$. Thus, if $D^{5}f(\,.\,)$ is bounded, it follows that $D^{4}P(\,.\,,\lambda,\sigma)$ is Lipschitz continuous, and we can apply Algorithm 1 to (\ref{eq:5.5})-(\ref{eq:5.6}), solving (\ref{eq:2.6}) with a fourth-order tensor method $\mathcal{M}_{4}$. Specifically, if $\mathcal{M}_{4}$ satisfies A6 for $p=4$, then (as in Theorem \ref{thm:4.3}) we conclude that Algorithm 1 applied to (\ref{eq:5.5})-(\ref{eq:5.6}) takes at most $\mathcal{O}\left(|\log(\epsilon)|^{2}\epsilon^{-\frac{5}{4}}\right)$ inner iterations of $\mathcal{M}_{4}$ if $\left\{\sigma^{(k)}\right\}$ is bounded, and $\mathcal{O}\left(\epsilon^{-\left(\frac{4}{\alpha-1}+\frac{5}{4}\right)}\right)$ inner iterations of $\mathcal{M}_{4}$ if $\left\{\sigma^{(k)}\right\}$ is unbounded. 
\end{remark}

\section{Conclusion}

In this paper, we have studied the worst-case complexity of an inexact Augmented Lagrangian method for constrained nonconvex optimization problems. For the case in which the penalty parameters are bounded, we established a complexity bound of $\mathcal{O}(|\log(\epsilon)|)$ outer iterations for the referred algorithm to generate an $\epsilon$-approximate KKT point, for $\epsilon\in (0,1)$. For the case in which the penalty parameters are unbounded, we proved an iteration complexity bound of $\mathcal{O}\left(\epsilon^{-2/(\alpha-1)}\right)$, where $\alpha>1$ controls the rate of increase of the penalty parameters. In the particular class of linearly constrained problems, these bounds yield to evaluation complexity bounds of $\mathcal{O}(\sigma_{max}\sqrt{n}|\log(\epsilon)|^{2}\epsilon^{-2})$ and $\mathcal{O}\left(\sqrt{n}\epsilon^{-\left(\frac{2(2+\alpha)}{\alpha-1}+2\right)}\right)$, respectively, when appropriate first-order methods ($p=1$) are used to approximately solve the unconstrained subproblems at each iteration. For problems having only linear equality constraints, these bounds are improved to
$\mathcal{O}(|\log(\epsilon)|^{2}\epsilon^{-(p+1)/p})$ and $\mathcal{O}\left(\epsilon^{-\left(\frac{4}{\alpha-1}+\frac{p+1}{p}\right)}\right)$, respectively, when appropriate $p$-order me\-thods ($p\geq 2$) are used as inner solvers.

A key point in the Augmented Lagrangian method considered in this work is that it requires the feasibility of the starting point $x_{0}$, which may be difficult to compute for general nonconvex constraints. Up to now, it is not clear how this restriction can be avoided. Another natural question is whether our analysis can be adapted in order to cover possibly less aggressive update rules for the penalty parameters, such as
\begin{equation*}
\sigma^{(k+1)}=\left\{\begin{array}{ll} \sigma^{(k),}&\text{if}\quad\theta^{(k+1)}\leq\gamma\theta^{(k)},\\
\alpha \sigma^{(k)},&\text{otherwise.}
\end{array}
\right.
\end{equation*}
The authors are planning to address these and other interesting questions in their future research.

\section*{Acknowledgments}
We are very grateful to Coralia Cartis and the two anonymous referees, whose comments helped to improve significantly the paper. We are also grateful to Mario Mart\'inez and Stephen Wright for their insightful comments on the first version of this work.

\section*{Funding}

G.N. Grapiglia was partially supported by the National Council for Scientific and Technological Development - Brazil (grant 406269/2016-5). Y. Yuan was partially supported by NSFC (grant 11688101).

\end{document}